\documentclass{amsart}
\usepackage{graphicx, color}
\usepackage{amscd}
\usepackage{amsmath,empheq}
\usepackage{amsfonts}
\usepackage{amssymb}
\usepackage{mathrsfs}
\usepackage[all]{xy}
\numberwithin{equation}{section}
\newtheorem{theorem}{Theorem}[section]

\newtheorem{lemma}[theorem]{Lemma}

\newtheorem{proposition}[theorem]{Proposition}

\newtheorem{definition}[theorem]{Definition}
\newenvironment{proof-sketch}{\noindent{\bf Sketch of Proof}\hspace*{1em}}{\qed\bigskip}

\everymath{\displaystyle}
\newcommand{\RR}{\mathbb{R}}

\newcommand{\intom}{\int_\Omega}

\newcommand{\di}{\displaystyle}
\newcommand{\ep}{\varepsilon}
\newcommand{\ee}{\mathcal E}
\newcommand{\nn}{\mathcal N}

\newcommand{\ri}{\rightarrow}

\newcommand{\sob}{W_0^{1,p_2(x)}(\Omega)}
\newcommand{\bb}{\begin{equation}}
\newcommand{\bbb}{\end{equation}}

\newcommand{\WW}{{\mathcal W}}

\begin{document}

\title[Double phase problems with variable growth]{Double phase problems with variable growth}
\author[M. Cencelj]{Matija Cencelj}
\address[M. Cencelj]{Faculty of Education and Faculty of Mathematics and Physics, University of Ljubljana \& Institute of Mathematics, Physics and Mechanics, 1000 Ljubljana, Slovenia}
\email{\tt matija.cencelj@guest.arnes.si}
\author[V.D. R\u{a}dulescu]{Vicen\c{t}iu D. R\u{a}dulescu}
\address[V.D. R\u{a}dulescu]{Faculty of Applied Mathematics, AGH University of Science and Technology, al. Mickiewicza 30, 30-059 Krak\'ow, Poland \& Institute of Mathematics, Physics and Mechanics, 1000 Ljubljana, Slovenia \& Institute of Mathematics ``Simion Stoilow" of the Romanian Academy of Sciences, P.O. Box 1-764,
          014700 Bucharest, Romania}
\email{\tt vicentiu.radulescu@imar.ro}
\author[D.D. Repov\v{s}]{Du\v{s}an D. Repov\v{s}}
\address[D.D. Repov\v{s}]{Faculty of Education and Faculty of Mathematics and Physics, University of Ljubljana \& Institute of Mathematics, Physics and Mechanics, 1000 Ljubljana, Slovenia}
\email{\tt dusan.repovs@guest.arnes.si}
\begin{abstract}
We consider a class of double phase variational integrals driven by nonhomogeneous potentials. We study the associated Euler equation and we highlight the existence of two different Rayleigh quotients. One of them is in relationship with the existence of an infinite interval of eigenvalues while the second one is associated with the nonexistence of eigenvalues. The notion of eigenvalue is understood in the sense of pairs of nonlinear operators, as introduced by Fu\v{c}i­k, Ne\v{c}as, Sou\v{c}ek, and  Sou\v{c}ek.  The analysis developed in this paper extends the abstract framework corresponding to some standard cases associated to the $p(x)$-Laplace operator, the generalized mean curvature operator, or the capillarity differential operator with variable exponent. The results contained in this paper complement the pioneering contributions of Marcellini, Mingione {\it et al.} in the field of variational integrals with unbalanced growth.
\end{abstract}
\maketitle
\begin{center}{\small\it This paper is dedicated with esteem and gratitude to Professor Carlo Sbordone on the occasion of his 70th birthday}\end{center}

\section{Unbalanced problems {\it \`a la} Marcellini and Mingione}\label{sect1}
The study of differential equations and variational problems involving variable
growth conditions was motivated by their various applications. In 1920, Bingham was
surprised to discover that some paints do not run like honey. He studied such
a behavior and described a strange phenomenon. There are fluids that first
flow, then stop spontaneously (Bingham fluids). Inside them, the forces that
create the flows do not reach a threshold. As this threshold is not reached, the fluid
flow deforms as a solid. Invented in the 17th century, the ``Flemish medium"
makes painting oil thixotropic: it flows under pressure of the brush, but freezes
as soon as you leave it to rest. While the exact composition of the Flemish
medium remains unknown, it is known that the bonds form gradually between
its components, which is why the picture freezes in a few minutes. Thanks to
this wonderful medium, Rubens was able to paint {\it La Kermesse} in only 24
hours.

The recent systematic study of nonlinear problems with variable exponents
is motivated by the description of several relevant phenomena that arise in the applied sciences. For instance, this mechanism can be used to give models for non-Newtonian fluids
that change their viscosity in the presence of an electro-magnetic field, which in this case
influences the size of the variable exponent, see Halsey \cite{hal} and R\r{u}\v{z}i\v{c}ka \cite{R}. Similar models appear in image segmentation, see Chen, Levine and Rao \cite{CLR}. Their framework is a combination of the Gaussian smoothing and regularization based on the
total variation.

In this paper, we are concerned with the study of a nonlinear problem whose features are the following:\\ (i) the presence of several differential operators with different growth, which generates a {\it double phase} associated energy;\\ (ii) the presence of several {\it variable} potentials, which involves a different environment (according to the point) as well as a degenerate abstract setting.

These attributes of the present paper imply multiple combined effects of the multiple variable exponents, the reaction term, and the values of a suitable real parameter. In particular, the existence of several Rayleigh quotients implies the existence of solutions in the case of high perturbations, while nontrivial solutions do not exist in the case of small perturbations of the reaction term. This is in accordance with respect to the position of a suitable parameter with respect to these Rayleigh quotients.

We recall in what follows some of the outstanding contributions of the Italian school (Marcellini, Mingione, Colombo, Baroni, {\it et al.}) to the study of unbalanced integral functionals and double phase problems. In the next section of this paper, we recall some basic properties of the Lebesgue and Sobolev spaces with variable exponent. We also refer to a new nonhomogeneous differential operator, which will be used in the present paper in the abstract setting of double phase problems with variable exponent. This operator has a broad spectrum and it extends not only the $p(x)$--Laplace operator but also the generalized mean curvature operator, the capillarity operator with variable exponent and other non-homogeneous differential operators.
The main results and proofs and developed in the remaining section of this paper.

This paper was motivated by several recent contributions to the qualitative analysis of nonlinear problems with unbalanced growth. We first refer to the pioneering contributions of Marcellini \cite{marce1,marce2,marce3} who studied lower semicontinuity and regularity properties of minimizers of certain quasiconvex integrals. Problems of this type arise in nonlinear elasticity and are connected with the deformation of an elastic body, cf. Ball \cite{ball1,ball2}. We also refer to Fusco and Sbordone \cite{fusco} for the study of regularity of minima of anisotropic integrals.

In order to recall the roots of double phase problems, let us assume that $\Omega$ is a bounded domain in $\RR^N$ ($N\geq 2$) with smooth boundary. If $u:\Omega\to\RR^N$ is the displacement and if $Du$ is the $N\times N$  matrix of the deformation gradient, then the total energy can be represented by an integral of the type
\bb\label{paolo}I(u)=\intom f(x,Du(x))dx,\bbb
where the energy function $f=f(x,\xi):\Omega\times\RR^{N\times N}\to\RR$ is quasiconvex with respect to $\xi$, see Morrey \cite{morrey}. One of the simplest examples considered by Ball is given by functions $f$ of the type
$$f(\xi)=g(\xi)+h({\rm det}\,\xi),$$
where ${\rm det}\,\xi$ is the determinant of the $N\times N$ matrix $\xi$, and $g$, $h$ are nonnegative convex functions, which satisfy the growth conditions
$$g(\xi)\geq c_1\,|\xi|^p;\quad\lim_{t\to+\infty}h(t)=+\infty,$$
where $c_1$ is a positive constant and $1<p<N$. The condition $p\leq N$ is necessary to study the existence of equilibrium solutions with cavities, that is, minima of the integral \eqref{paolo} that are discontinuous at one point where a cavity forms; in fact, every $u$ with finite energy belongs to the Sobolev space $W^{1,p}(\Omega,\RR^N)$, and thus it is a continuous function if $p>N$. In accordance with these problems arising in nonlinear elasticity, Marcellini \cite{marce1,marce2} considered continuous functions $f=f(x,u)$ with {\it unbalanced growth} that satisfy
$$c_1\,|u|^p\leq |f(x,u)|\leq c_2\,(1+|u|^q)\quad\mbox{for all}\ (x,u)\in\Omega\times\RR,$$
where $c_1$, $c_2$ are positive constants and $1\leq p\leq q$. Regularity and existence of solutions of elliptic equations with $p,q$--growth conditions were studied in \cite{marce2}.

The study of non-autonomous functionals characterized by the fact that the energy density changes its ellipticity and growth properties according to the point has been continued in a series of remarkable papers by Mingione {\it et al.} \cite{mingi1}--\cite{mingi5}. These contributions are in relationship with the works of Zhikov \cite{zhikov1}, in order to describe the
behavior of phenomena arising in nonlinear
elasticity.
In fact, Zhikov intended to provide models for strongly anisotropic materials in the contect of homogenisation.
These functionals revealed to be important also in the study of duality theory
and in the context of the Lavrentiev phenomenon \cite{zhikov2}. In particular, Zhikov considered three different model
functionals for this situation in relation to the Lavrentiev phenomenon. These are
\bb\label{mingfunc}\begin{array}{ll}
{\mathcal M}(u)&\di :=\intom c(x)|Du|^2dx,\quad 0<1/c(\cdot)\in L^t(\Omega),\ t>1\\
{\mathcal V}(u)&\di :=\intom |Du|^{p(x)}dx,\quad 1<p(x)<\infty\\
{\mathcal P}_{p,q}(u)&\di :=\intom (|Du|^p+a(x)|Du|^q)dx,\quad 0\leq a(x)\leq L,\ 1<p<q.
\end{array}
\bbb

The functional ${\mathcal M}$ is well-known and there is a loss of ellipticity on the set $\{x\in\Omega;\ c(x)=0\}$. This functional has been studied at length in the context of equations involving
Muckenhoupt weights. The functional ${\mathcal V}$ has also been the object of intensive interest nowadays and a
huge literature was developed on it.  We  refer to Acerbi and Mingione \cite{acerbi2} for gradient estimates and pioneering contributions to the qualitative analyze of minimizers of nonstandard energy functionals involving variable exponents. We refer to Colombo and Mingione \cite{colombo} as a reference paper on double phase problems studied independently of their variational structure.
We also point out the abstract setting, respectively the variational analysis developed in the monographs  by Diening,  Harjulehto, H\"{a}st\"{o}, and  R\r{u}\v{z}i\v{c}ka \cite{diening}, respectively by R\u adulescu and Repov\v s \cite{radrep}.
The energy functional defined by ${\mathcal V}$ was used to build models for strongly
anisotropic materials: in a material made of different components, the exponent $p(x)$
dictates the geometry of a composite that changes its hardening exponent according to
the point.  The functional ${\mathcal P}_{p,q}$ defined in \eqref{mingfunc} appears as un upgraded version of ${\mathcal V}$. Again, in this
case, the modulating coefficient $a(x)$ dictates the geometry of the composite made by
two differential materials, with hardening exponents $p$ and $q$, respectively.

The functionals displayed in \eqref{mingfunc} fall in the realm of the so-called functionals with
nonstandard growth conditions of $(p, q)$--type, according to Marcellini's terminology. These are functionals of the type in \eqref{paolo}, where the energy density satisfies
$$|\xi|^p\leq f(x,\xi)\leq  |\xi|^q+1,\quad 1\leq p\leq q.$$

Another significant model example of a functional with $(p,q)$--growth studied by Mingione {\it et al.} \cite{mingi1}--\cite{mingi5} is given by
$$u\mapsto \intom |Du|^p\log (1+|Du|)dx,\quad p\geq 1,$$
which is a logarithmic perturbation of the $p$-Dirichlet energy.

General models with $(p,q)$-growth in the context of
geometrically constrained problems have been recently studied by De Filippis \cite{cristina}. This seems to be the first work dealing with $(p,q)$-conditions with manifold constraint. Refined regularity results
are proved
in \cite{cristina},
by using an approximation technique relying on estimates obtained through a careful use of difference quotients. A key role is played by the method developed by
Esposito, Leonetti, and Mingione \cite{esposito} in order to prove the equivalence between the absence of Lavrentiev phenomenon and the extra regularity of the minimizers for unconstrained, non-autonomous variational problems.

The main feature of this paper is the study of a class of unbalanced double phase problems with variable exponent. In such a way, the present paper complements our previous related contributions to this field, see \cite{dou1}, \cite{dou2}. The present paper extends and complements the main results obtained in \cite{rjam} and \cite{chorfi} (see also \cite[Section 3.3]{radrep}).

\section{Spaces and operators with variable exponent}\label{sect2}
Nonlinear problems with non-homogeneous structure are motivated by several models in mathematical physics and other applied sciences that are described by partial differential equations with one or more variable exponents. In some circumstances, the standard analysis based on the theory of usual Lebesgue and Sobolev function  spaces,  $L^p$ and $W^{1,p}$, is not appropriate in the framework of material that involve non-homogeneities.
The presence of a variable exponent allows to describe in a proper and accurate manner the geometry of a material which is allowed to change its hardening  exponent according to the point.
This leads to the analysis of variable exponents Lebesgue and Sobolev function spaces (denoted by $L^{p(x)}$ and $W^{1,p(x)}$), where $p$ is a real-valued (non-constant) function.

\subsection{Lebesgue and Sobolev spaces with variable potential}
 Let $S(\Omega )$ be the set of all measurable real
valued functions defined on $\Omega $. Let
\begin{equation*}
C_{+}(\overline{\Omega })=\left\{ u;\ u\in C(\overline{\Omega })%
\text{, }u(x)>1\text{ for }x\in \overline{\Omega } \right\} ,
\end{equation*}%
\begin{equation*}
L^{p(\cdot )}(\Omega )=\left\{ u\in S(\Omega );\ \int_{\Omega }|
u(x)|^{p(x)}dx<\infty \right\} .
\end{equation*}

For all $p\in C_+(\Omega)$ we define
$$p^+=\sup_{x\in\Omega}p(x)\quad\mbox{and}\quad p^-=\inf_{x\in\Omega}p(x).$$

The function space $L^{p(\cdot )}(\Omega )$ is equipped with the Luxemburg norm%
\begin{equation*}
\left\vert u\right\vert _{L^{p(\cdot )}(\Omega )}=\inf \left\{ \lambda
>0;\ \int_{\Omega }\left\vert \frac{u(x)}{\lambda }\right\vert
^{p(x)}dx\leq 1 \right\} .
\end{equation*}

Then ($L^{p(\cdot )}(\Omega )$, $\left\vert \cdot \right\vert
_{L^{p(\cdot )}(\Omega )}$) becomes a Banach space, and we call it a variable exponent
Lebesgue space.

The following properties of spaces with variable exponent (Propositions \ref{prop2.1}--\ref{prop2.4}) are essentially due to Fan and Zhao \cite{cite7}, see also  \cite{diening}, \cite{radrep}, and \cite{radnla}.
We also refer to the important contributions of Edmunds {\it et al.}, see \cite{edm2}--\cite{edm3}.

\begin{proposition}\label{prop2.1} The space $%
(L^{p(\cdot )}(\Omega ),\left\vert u\right\vert _{L^{p(\cdot )}(\Omega )})$
is a separable, uniformly convex Banach space, and its conjugate space is $%
L^{p^{\prime }(\cdot )}(\Omega )$, where $\frac{1}{p(x)}+\frac{1}{p^{\prime
}(x)}=1$. For any $u\in L^{p(\cdot )}(\Omega )$ and $v\in L^{p^{\prime
}(\cdot )}(\Omega )$, we have the following H\"older inequality
\begin{equation*}
\left\vert \int_{\Omega }uvdx\right\vert \leq \left(\frac{1}{p^{-}}+\frac{1}{%
p^{\prime -}}\right)\left\vert u\right\vert _{L^{p(\cdot )}(\Omega )}\left\vert
v\right\vert _{L^{p^{\prime }(\cdot )}(\Omega )}.
\end{equation*}
\end{proposition}

\begin{proposition}\label{prop2.2} If $f:$ $\Omega
\times
\mathbb{R}
\rightarrow
\mathbb{R}
$ is a Carath\'{e}odory function and satisfies%
\begin{equation*}
\left\vert f(x,s)\right\vert \leq d(x)+b\left\vert s\right\vert
^{p_{1}(x)/p_{2}(x)}\text{ for any }x\in \Omega ,s\in
\mathbb{R}
,
\end{equation*}%
where $p_{1}$, $p_{2}\in C_{+}(\overline{\Omega })$, $d(x)\in L^{p_{2}(\cdot
)}(\Omega )$, $d(x)\geq 0$, $b\geq 0$, then the Nemytsky operator from $%
L^{p_{1}(\cdot )}(\Omega )$ to $L^{p_{2}(\cdot )}(\Omega )$ defined by $%
(N_{f}u)(x)=f(x,u(x))$ is a continuous and bounded operator. \end{proposition}

\begin{proposition}\label{prop2.3} If we denote
\begin{equation*}
\rho _{p(\cdot )}(u)=\int_{\Omega }\left\vert u\right\vert ^{p(x)}dx\text{, }%
\forall u\in L^{p(\cdot )}(\Omega ),
\end{equation*}%
then the following properties hold:

i) $\left\vert u\right\vert _{L^{p(\cdot )}(\Omega
)}<1\ (=1;>1)\Longleftrightarrow \rho _{p(\cdot )}(u)<1\ (=1;>1);$

ii) $\left\vert u\right\vert _{L^{p(\cdot )}(\Omega )}>1\Longrightarrow
\left\vert u\right\vert _{L^{p(\cdot )}(\Omega )}^{p^{-}}\leq \rho _{p(\cdot
)}(u)\leq \left\vert u\right\vert _{L^{p(\cdot )}(\Omega )}^{p^{+}};$

$\left\vert u\right\vert _{L^{p(\cdot )}(\Omega )}<1\Longrightarrow
\left\vert u\right\vert _{L^{p(\cdot )}(\Omega )}^{p^{-}}\geq \rho _{p(\cdot
)}(u)\geq \left\vert u\right\vert _{L^{p(\cdot )}(\Omega )}^{p^{+}};$

iii) $\left\vert u\right\vert _{L^{p(\cdot )}(\Omega )}\rightarrow \infty
\Longleftrightarrow \rho _{p(\cdot )}(u)\rightarrow \infty .$\end{proposition}

\begin{proposition}\label{prop2.4} If $u$, $u_{n}\in
L^{p(\cdot )}(\Omega )$, $n=1,2,\cdots ,$ then the following statements are
equivalent:

1) $\underset{k\rightarrow \infty }{\lim }$ $\left\vert u_{k}-u\right\vert
_{L^{p(\cdot )}(\Omega )}=0;$

2) $\underset{k\rightarrow \infty }{\lim }$ $\rho _{p(\cdot )}\left(
u_{k}-u\right) =0;$

3) $u_{k}\rightarrow u$ in measure in $\Omega $ and $\underset{k\rightarrow
\infty }{\lim }$ $\rho _{p(\cdot )}\left( u_{k}\right) =\rho _{p(\cdot
)}\left( u\right) $.\end{proposition}

The variable exponent Sobolev space $W^{1,p(\cdot )}(\Omega )$ is defined by%
\begin{equation*}
W^{1,p(\cdot )}(\Omega )=\left\{ u\in L^{p(\cdot )}\left( \Omega \right)
;\ D u\in \lbrack L^{p(\cdot )}\left( \Omega \right)
]^{N} \right\} ,
\end{equation*}%
and it is equipped with the norm
\begin{equation*}
\left\Vert u\right\Vert _{W^{1,p(\cdot )}(\Omega )}=\left\vert u\right\vert
_{L^{p(\cdot )}(\Omega )}+\left\vert D u\right\vert _{L^{p(\cdot
)}(\Omega )},\quad\mbox{for all}\ u\in W^{1,p(\cdot )}\left( \Omega \right) .
\end{equation*}

We denote by $W_{0}^{1,p(\cdot )}(\Omega )$ the closure of $C_{0}^{\infty
}\left( \Omega \right) $ in $W^{1,p(\cdot )}\left( \Omega \right) $. The dual space of $W_0^{1,p(x)}(\Omega)$ is denoted by $W^{-1,p'(x)}(\Omega)$, where $p'(x)$ is the conjugate exponent of $p(x)$.

We recall that the critical Sobolev exponent is defined as follows:
 $$p^*(x)=
\left\{
\begin{array}{lll}
&\displaystyle\frac{Np(x)}{N-p(x)},&\quad\mbox{if}\ p(x)<N,\\
&\displaystyle+\infty, &\quad\mbox{if}\ p(x)\geq N.
\end{array}
\right.
$$

We point out that if
 $q \in C^+(\overline{\Omega})$ and $q(x) \leq p^*(x)$ for all $x\in\overline\Omega$, then $W^{1,p(\cdot )}(\Omega )$ is
continuously  embedded in $L^{q(\cdot )}(\Omega)$. This embedding is compact if $$\inf\{p^*(x)-q(x);\ x\in\Omega\}>0.$$

The  Lebesgue and Sobolev
spaces with variable exponents coincide with the usual Lebes\-gue and Sobolev spaces provided that $p$ is constant. According to \cite[pp. 8-9]{radrep}, these function spaces $L^{p(x)}$ and $W^{1,p(x)}$ have some unusual properties,
such as:

(i)
Assuming that $1<p^-\leq p^+<\infty$ and $p:\overline\Omega\rightarrow [1,\infty)$ is a smooth function, then the following co-area formula
$$\int_\Omega |u(x)|^pdx=p\int_0^\infty t^{p-1}\,|\{x\in\Omega ;\ |u(x)|>t\}|\,dt$$
has  no analogue in the framework of variable exponents.

(ii) Spaces $L^{p(x)}$ do {\it not} satisfy the {\it mean continuity property}. More exactly, if $p$ is nonconstant and continuous in an open ball $B$, then there is some $u\in L^{p(x)}(B)$ such that $u(x+h)\not\in L^{p(x)}(B)$ for every $h\in{\mathbb R}^N$ with arbitrary small norm.

(iii) Function spaces with variable exponent
 are {\it never} invariant with respect to translations.  The convolution is also limited. For instance,  the classical Young inequality
$$| f*g|_{p(x)}\leq C\, | f|_{p(x)}\, \| g\|_{L^1}$$
remains valid if and only if
$p$ is constant.

\subsection{A generalized operator with variable exponent}
Assume that $p\in C_+(\overline\Omega)$.
Kim and Kim \cite{kim} introduced an important class of nonhomogeneous operators with variable exponent. These are operators of the type
$${\rm div}\, (\phi(x,|Du|)Du),$$
where $\phi(x,\xi)$ satisfies the following hypotheses:

\smallskip
\noindent ($\phi$1) the mapping $\phi(\cdot,\xi)$ is measurable on $\Omega$ for all $\xi\geq 0$ and it is locally absolutely continuous on $[0,\infty)$ for almost all $x\in\Omega$;

\smallskip\noindent ($\phi$2) there exist  $a\in L^{p'(x)}(\Omega)$  and $b>0$ such that
$$|\phi (x,|v|)v|\leq a(x)+b|v|^{p(x)-1},$$
for almost all $x\in\Omega$ and for all $v\in\RR^N$;

\smallskip\noindent ($\phi$3) there exists $c>0$ such that
$$\phi(x,\xi)\geq c\xi^{p(x)-2},\quad \phi(x,\xi)+\xi\frac{\partial\phi}{\partial\xi}(x,\xi)\geq c\xi^{p(x)-2}$$
for almost all $x\in\Omega$ and for all $\xi>0$.



\smallskip
We now give some examples of potentials satisfying hypotheses ($\phi$1)--($\phi$3). If $\phi (x,\xi)=\xi^{p(x)-2}$ then we obtain the standard $p(x)$-Laplace operator, that is, $$\Delta_{p(x)}u:={\rm div}\, (|D  u|^{p(x)-2}D  u).$$
The abstract setting considered in this paper includes the case $\phi (x,\xi)=(1+|\xi|^2)^{(p(x)-2)/2}$, which corresponds to the generalized mean curvature operator
$${\rm div}\, \left[(1+|D  u|^2)^{(p(x)-2)/2}D  u \right].$$
The capillarity equation corresponds to
$$\phi(x,\xi)=\left(1+\frac{\xi^{p(x)}}{\sqrt{1+\xi^{2p(x)}}}\right)\xi^{p(x)-2},\quad x\in\Omega,\ \xi>0,$$ hence the corresponding capillary phenomenon is described by the differential operator
$${\rm div}\, \left[\left( 1+\frac{|D  u|^{p(x)}}{\sqrt{1+|D  u|^{2p(x)}}} \right)|D  u|^{p(x)-2}D  u\right].$$

\smallskip
For $\phi$  described in hypotheses ($\phi$1)--($\phi$3) we set
$$\Phi_0(x,t):=\int_0^t\phi (x,s)s ds$$
and define the functional $\Phi:W_0^{1,p(x)}(\Omega)\to\RR$ by
$$\Phi(u)=\intom \Phi_0(x,|Du(x)|)dx.$$

The following result was established in \cite[Lemma 3.2]{kim}.

\begin{proposition}\label{l3.2kim} Assume that ($\phi$1) and ($\phi$2) hold. Then $\Phi\in W_0^{1,p(x)}(\Omega)$ and its G\^ateaux derivative is given by
$$\langle \Phi'(u),v\rangle=\intom\phi(x,|Du(x)|)Du(x)Dv(x)dx\quad\mbox{for all}\ u,v\in W_0^{1,p(x)}(\Omega).$$
\end{proposition}

The following estimate plays a crucial role in order to establish that $\Phi'$ is an operator of type $(S)_+$. In the particular case when $p(x)$ is a constant, this result is usually referred as Simon's inequality \cite[formula 2.2]{simon} (see also \cite[p. 713]{fpr2008}).

Denote
$$\Omega_1=\{x\in\Omega;\ 1<p(x)<2\},\quad \Omega_2=\{x\in\Omega;\ p(x)\geq 2\}$$
(we allow the possibility that one of these sets is empty).

 The following result was established in
\cite[Proposition 3.3]{kim}.

\begin{proposition}\label{p3.3kim} Assume that hypotheses ($\phi$1)--($\phi$3) are fulfilled. Then the following estimate holds
$$\begin{array}{ll}
&\di \langle\phi(x,|u|)u-\phi(x,|v|)v,u-v\rangle\geq\\
&\di \left\{\begin{array}{lll}
&\di c(|u|+|v|)^{p(x)-2}|u-v|^2&\quad\mbox{if}\ x\in\Omega_1,\ (u,v)\not=(0,0)\\
&\di 4^{1-p^+}c|u-v|^{p(x)}&\quad\mbox{if}\ x\in\Omega_2,
\end{array}\right.
\end{array}$$
where $c$ is the positive constant from ($\phi$3).
\end{proposition}

For the following result we refer to \cite[Lemma 3.4]{kim}.

\begin{proposition}\label{l3.4kim} Assume that hypotheses ($\phi$1)--($\phi$3) are fulfilled. Then the operator $\Phi':W_0^{1,p(x)}(\Omega)\to W^{-1,p'(x)}(\Omega)$ is a strictly monotone mapping of type $(S)_+$, that is, if $$u_n\rightharpoonup u\ \mbox{in}\ W_0^{1,p(x)}(\Omega)\ \mbox{as}\ n\to\infty\ \mbox{and}\ \limsup_{n\to\infty}\langle \Phi'(u_n)-\Phi'(u),u_n-u\rangle\leq 0,$$ then $u_n\to u$ in $W_0^{1,p(x)}(\Omega)$ as $n\to\infty$.
\end{proposition}

We refer to Baraket,  Chebbi,  Chorfi, and  R\u adulescu \cite{chorfi} for the mathematical analysis of a nonlinear Dirichlet problem driven by operators as described in this section.

\section{Hypotheses and main results}\label{sect3}
We study of the following nonlinear problem
\begin{equation}\label{1}
\left\{\begin{array}{lll}
& -{\rm div}\, (\phi(x,|D  u|)D  u)-{\rm div}\, (\psi(x,|D  u|)D  u)+w(x)\theta(x,|u|)u=\\
&\lambda\,(|u|^{r-2}u+|u|^{s-2}u)\quad\mbox{in}\ \Omega\\
& u=0\quad \mbox{on}\ \partial\Omega\,,
\end{array}\right.
\end{equation}
where $\Omega$ is a smooth bounded domain in $\RR^N$ ($N\geq 2$), $w$ is an indefinite potential, and $\lambda$ is a  real parameter.

We assume that $\phi$, $\psi$ satisfy the following hypotheses:

\smallskip
\noindent (H1) the mappings $\phi(\cdot,\xi)$, $\psi(\cdot,\xi)$, and $\theta(\cdot,\xi)$ are measurable on $\Omega$ for all $\xi\geq 0$ and they are locally absolutely continuous on $[0,\infty)$ for almost all $x\in\Omega$;

\smallskip\noindent (H2) there exist  $a_1\in L^{p_1'(x)}(\Omega)$, $a_2\in L^{p_2'(x)}(\Omega)$,   and $b>0$ such that
for almost all $x\in\Omega$ and for all $v\in\RR^N$
$$|\phi (x,|v|)v|\leq a_1(x)+b|v|^{p_1(x)-1},$$
$$|\psi (x,|v|)v|\leq a_2(x)+b|v|^{p_2(x)-1}$$
and
$$|\theta (x,|v|)v|\leq b|v|^{p_3(x)-1};$$

\smallskip\noindent (H3) there exists $c>0$ such that for almost all $x\in\Omega$ and for all $\xi>0$
$$\min\left\{\phi(x,\xi), \phi(x,\xi)+\xi\frac{\partial\phi}{\partial\xi}(x,\xi)\right\}\geq c\xi^{p_1(x)-2},$$
$$\min\left\{\psi(x,\xi), \psi(x,\xi)+\xi\frac{\partial\psi}{\partial\xi}(x,\xi)\right\}\geq c\xi^{p_2(x)-2}$$
and
$$\theta(x,\xi)\geq c\xi^{p_3(x)-2}.$$

Throughout this paper, we assume that $r$ and $s$ are real numbers while $p_1$, $p_2$, and $p_3$ are continuous on $\overline\Omega$ such that
\begin{equation}\label{2}
1<\max_{x\in\overline\Omega}p_1(x)<r\leq\min_{x\in\overline\Omega}p_3(x)\leq
\max_{x\in\overline\Omega}p_3(x)
\leq s<\min_{x\in\overline\Omega}p_2(x).
\end{equation}
We work in a subcritical setting, which is defined by the hypothesis
\begin{equation}\label{3}
\min_{x\in\overline\Omega}\{ p_1^*(x)-p_2(x)\}>0,
\end{equation}
where
$$p_1^*(x):=\left\{\begin{array}{lll}
&\di \frac{Np_1(x)}{N-p_1(x)}\qquad &\mbox{if $p_1(x)<N$}\\
& +\infty\qquad &\mbox{if $p_1(x)\geq N$.}\end{array}\right.$$

We also assume that the weight $w$ is indefinite (that is, sign-changing) and  $w\in L^\infty(\Omega)$.

In accordance with the comments developed in the previous section, we introduce the functions
$$\Phi_0(x,t):=\int_0^t\phi (x,s)s ds;\quad \Psi_0(x,t):=\int_0^t\psi (x,s)s ds;\quad \Theta_0(x,t):=\int_0^t\theta (x,s)s ds$$
and define the functionals $$\Phi:W_0^{1,p_1(x)}(\Omega)\to\RR;\quad
\Phi(u)=\intom \Phi_0(x,|Du(x)|)dx$$
$$\Psi:W_0^{1,p_2(x)}(\Omega)\to\RR;\quad
\Psi(u)=\intom \Psi_0(x,|Du(x)|)dx$$
and
$$\Theta:W_0^{1,p_3(x)}(\Omega)\to\RR;\quad
\Theta(u)=\intom w(x)\Theta_0(x,|u(x)|)dx.$$

An important role in the proof of our main result is played by the following assumption, which is also
used in \cite{kim} for proving the existence of weak solutions in a different framework:

\smallskip\noindent (H4) For all $x\in\overline\Omega$ and all $\xi\in\RR^N$, the following estimate holds:
$$0\leq[\phi(x,|\xi|)+\psi(x,|\xi|)]\,|\xi|^2\leq p_1^+[\Phi_0(x,|\xi|)+\Psi_0(x,|\xi|)].$$

Taking into account the growth of the potentials $\phi$, $\psi$ and $\theta$ defined in hypothesis \eqref{2}, it follows that the natural function space for the existence of all energies $\Phi$, $\Psi$ and $\Theta$ is $\sob$.

\begin{definition}\label{defi1}
We say that $u\in\sob\setminus\{0\}$ is a solution of problem \eqref{1} if
$$\begin{array}{ll}
&\di\intom \left[\phi(x,|D  u|)+\psi(x,|D  u|) \right]DuDvdx+\intom w(x)\theta(x,|  u|)uvdx=\\ &\di
\lambda\intom (|u|^{r-2}+|u|^{s-2})uvdx\end{array}$$
for all $v\in\sob$.
\end{definition}

We will say that the corresponding real number $\lambda$  for which problem \eqref{1} has a nontrivial solution is an {\it eigenvalue}  while the corresponding $u\in\sob\setminus\{0\}$ is an {\it eigenfunction} of the problem. These terms are in accordance with the related notions introduced by Fu\v{c}ik, Ne\v{c}as, Sou\v{c}ek, and Sou\v{c}ek \cite[p. 117]{fucik} in the context of {\it nonlinear} operators. Indeed, if we denote $$A(u):=\Phi(u)+\Psi(u)+\Theta(u)\quad \mbox{and}\quad
B(u):= \intom\left(\frac{|u|^r}{r}+\frac{|u|^s}{s} \right)dx$$
then $\lambda$ is an eigenvalue for the pair $(A,B)$ of nonlinear operators (in the sense of \cite{fucik}) if and only if
there is a corresponding eigenfunction that is a solution of problem \eqref{1} as described in Definition \ref{defi1}.

 The energy functional associated to problem \eqref{1} is $\ee:\sob\to\RR$ defined by
$$\ee(u):=\Phi(u)+\Psi(u)+\Theta(u)-\lambda\intom\left(\frac{|u|^r}{r}+\frac{|u|^s}{s} \right)dx.$$

The contribution of the convection term $Du$ in the expression of the energy is given by
$$\Phi(u)+\Psi(u)=\intom \left( \Phi_0(x,|Du(x)|)+\Psi_0(x,|Du(x)|)\right)dx,$$
which corresponds to a {\it nonhomogeneous double phase problem}.

We describe in what follows the main result of this paper.
We first point out that hypothesis \eqref{2} implies that
 $\ee$ is coercive. Furthermore,
 the Rayleigh quotient associated to problem \eqref{1} has a blow-up behavior both at the origin and at infinity. More exactly, we have
$$\lim_{\|u\|_{p_2(x)}\rightarrow 0}\di\frac{\Phi(u)+\Psi(u)+\Theta(u)}{\di\frac 1r\int_\Omega|u|^{r}\;dx+
\di\frac 1s\int_\Omega|u|^{s}\;dx}=+\infty$$ and
$$\lim_{\|u\|_{p_2(x)}\rightarrow\infty}\di\frac{\Phi(u)+\Psi(u)+\Theta(u)}{\di\frac 1r\int_\Omega|u|^{r}\;dx+
\di\frac 1s\int_\Omega|u|^{s}\;dx}=+\infty\,,$$ where
$\|\cdot\|_{p_2(x)}$ denotes the norm in the  space $W_0^{1,p_2(x)}(\Omega)$.

Roughly speaking, these remarks show that the infimum of this (first) Rayleigh quotient associated to problem \eqref{1} is finite, that is,
$$\lambda^*:=\inf_{u\in\sob\setminus\{0\}}
\frac{\Phi(u)+\Psi(u)+\Theta(u)}{\di\frac 1r\int_\Omega|u|^{r}\;dx+
\di\frac 1s\int_\Omega|u|^{s}\;dx}\in\RR.$$

Our main result shows that problem \eqref{1} has a solution provided that $\lambda$ is at least $\lambda^*$. This corresponds to {\it high perturbations} (with respect to $\lambda$) of the reaction term in problem \eqref{1}.

In view of the non-homogeneous character of problem \eqref{1}, there is a second Rayleigh quotient, namely
$$\frac{\di \intom (\phi(x,|Du|)+\psi(x,|Du|))|Du|^2dx+\intom w(x)\theta (x,|u|)u^2dx}{\di\intom (|u|^r+|u|^s)dx}\,.$$

Let $\lambda_*$ denote the second critical parameter that corresponds to the infimum of
this new Rayleigh quotient, that is,
$$\lambda_*:=\inf_{u\in\sob\setminus\{0\}}\frac{\di \intom (\phi(x,|Du|)+\psi(x,|Du|))|Du|^2dx+\intom w(x)\theta (x,|u|)u^2dx}{\di\intom (|u|^r+|u|^s)dx}\,.$$

The existence of two Rayleigh quotients (and, consequently, of two different Dirichlet-type energies) is due to the presence of the nonstandard potential
$\phi(x,\xi)$. In the simplest case corresponding to $\phi(x,\xi)=\xi^{p(x)-2}$ the corresponding energies are $$\intom |Du|^{p(x)}dx\quad\mbox{and}\quad\intom\frac{1}{p(x)}\,|Du|^{p(x)}dx$$
with associated Euler equations
$${\rm div}\, (p(x)|Du|^{p(x)-2}Du)=0$$
and
$${\rm div}\, (|Du|^{p(x)-2}Du)=0.$$

Our main result also shows that problem \eqref{1} does not have solutions for the values of $\lambda$ less than $\lambda_*$. These properties are described in the following theorem, which asserts the existence of a continuous family of eigenvalues starting with the principal eigenvalue $\lambda^*$, while no eigenvalues exist below $\lambda_*$.

\begin{theorem}\label{t1}
Assume that hypotheses (H1)--(H4), \eqref{2} and \eqref{3} are satisfied. Then the following properties hold:

(i) $\lambda^*$ is an eigenvalue of problem \eqref{1};

(ii) all values $\lambda\geq\lambda^*$ are eigenvalues of problem \eqref{1};

(iii) any $\lambda<\lambda_*$ is not an eigenvalue of problem \eqref{1}.
\end{theorem}

This result shows that $\lambda^*$ plays an important role in the context of problem \eqref{1}. Next,
we focus on $\lambda^*$ and we look it as a quantity depending on the weight $w\in L^\infty(\Omega)$, hence
$$\lambda^*(w):=\inf_{u\in\sob}
\frac{\di\intom \Phi_0(x,|Du(x)|)dx+\intom \Psi_0(x,|Du(x)|)dx+\intom w(x)\Theta_0(x,|u(x)|)dx}{\di\frac 1r\int_\Omega|u|^{r}\;dx+
\di\frac 1s\int_\Omega|u|^{s}\;dx}.$$

We refer to Colasuonno and Squassina \cite{cola} who also studied eigenvalues of double phase problems and established the existence of  a sequence of nonlinear eigenvalues by a minimax procedure. 

Next, we are interested in the following optimization problem: is there $w_0\in L^\infty(\Omega)$ such that
$$\lambda^*(w_0)=\inf_{w\in L^\infty(\Omega)}\lambda^*(w)\ ?$$

The following result gives a positive answer to this question in the framework of bounded closed subsets of $L^\infty(\Omega)$.

\begin{theorem}\label{t2}
Assume that hypotheses (H1)--(H4), \eqref{2} and \eqref{3} are verified.
Let $\WW$ be a nonempty, bounded, closed subset of $L^\infty(\Omega)$.

Then there exists $w_0\in\WW$ such that
 $$\lambda^*(w_0)=\inf_{w\in \WW}\lambda^*(w).$$
\end{theorem}

The proof of Theorem \ref{t1} is given in section \ref{sect4} while section \ref{sect5}
is devoted to the proof of Theorem~\ref{t2}. Concluding remarks and some perspectives are given in the final part of this paper.

\section{Existence of an unbounded interval of eigenvalues}\label{sect4}
In this section we give the proof of Theorem \ref{t1}.

We denote
$$\ee_0(u):=\Phi(u)+\Psi(u);\quad
\ee_1(u):=\ee_0(u)+\Theta(u);\quad
\ee_2(u):=\intom\left(\frac{|u|^r}{r}+\frac{|u|^s}{s}\right)dx .$$
It follows that $\ee(u)=\ee_1(u)+\ee_2(u)$.

We denote by $\|\,\cdot\,\|$ the norm in the Banach space $\sob$.

The first result establishes a coercivity behavior of the energy $\ee_1$ with respect to $\ee_2$, both near the origin and at infinity.

\begin{lemma}\label{lema1}
We have $$\lim_{\|u\|\to 0}\frac{\ee_1(u)}{\ee_2(u)}=\lim_{\|u\|\to \infty}\frac{\ee_1(u)}{\ee_2(u)}=+\infty.$$
\end{lemma}

\begin{proof}
By \eqref{2}, we deduce that
$$\ee_2(u)\leq\intom (|u|^r+|u|^s)dx.$$
By Sobolev embeddings we have
$$\ee_2(u)\leq C_1(\|u\|^{r}+\|u\|^{s})\quad\mbox{for all}\ u\in\sob.$$

Next, we obtain by (H3)
$$\ee_0(u)\geq \frac{c}{p_1^+}\intom |Du|^{p_1(x)}dx+\frac{c}{p_1^+}\intom |Du|^{p_1(x)}dx=
C_2\intom (|Du|^{p_1(x)}+|Du|^{p_2(x)})dx.$$

We now estimate $\Theta(u)$ bu using hypotheses (H2) and \eqref{2}. We obtain
$$\begin{array}{ll}
\Theta(u)&\di =\intom w(x)\Theta_0(x,|u(x)|)dx=\intom w(x)\int_0^{|u(x)|}\theta (x,s)sds\\
&\geq -\|w\|_\infty\,\frac{b}{p_3^-}\intom |u|^{p_3(x)}dx \\
&\geq -C_3(\|u\|^r+\|u\|^s ).\end{array}$$
It follows that
\bb\label{above}\frac{\ee_1(u)}{\ee_2(u)}\geq\frac{\di C_2\intom (|Du|^{p_1(x)}+|Du|^{p_2(x)})dx}{C_1(\|u\|^{r}+\|u\|^{s})}-C_4.\bbb

Let us first assume that $(u_n)\subset\sob$ and $\|u_n\|\to 0$. Assuming that $\|u_n\|<1$ for all $n$, relation \eqref{above} implies that for all $n$
$$\frac{\ee_1(u_n)}{\ee_2(u_n)}\geq\frac{C_2\|u_n\|^{p_1^+}}{C_1(\|u_n\|^{r}+\|u_n\|^{s})}-C_4.$$
Since $p_1^+<r\leq s$ and $\|u_n\|\to 0$ we deduce that $\ee_1(u_n)/\ee_2(u_n)\to+\infty$ as $n\to\infty$.

To conclude the proof, let us now assume that $(u_n)\subset\sob$ is chosen arbitrarily so that $\|u_n\|\to \infty$. Without loss of generality we can assume that $\|u_n\|>1$ for all $n$. This time relation \eqref{above} yields
$$\frac{\ee_1(u_n)}{\ee_2(u_n)}\geq\frac{C_2\|u_n\|^{p_2^-}}{C_1(\|u_n\|^{r}+\|u_n\|^{s})}
-C_4.$$
By \eqref{2} we have $r\leq s<p_2^-$. Since $\|u_n\|\to \infty$, we obtain $$\lim_{n\to\infty}\frac{\ee_1(u_n)}{\ee_2(u_n)}=+\infty,$$ which concludes the proof of Lemma \ref{lema1}.
\end{proof}

For the following result we refer to \cite[Lemma 4.3]{kim}. This property establishes that $\Phi$ is  lower semicontinuous with respect to the weak topology of $\sob$.

\begin{lemma}\label{l4.3kim}
Assume that hypotheses (H1)--(H3) and \eqref{3} are fulfilled. Then $\Phi$ is weakly lower semicontinuous, that is, $u_n\rightharpoonup u$ in $\sob$ as $n\to\infty$ implies that $\Phi(u)\leq\liminf_{n\to\infty}\Phi(u_n)$.
\end{lemma}

The next result establishes that the infimum of $\ee_1(u)/\ee_2(u)$ is attained in $\sob\setminus\{0\}$ by an element which is an eigenfunction of problem \eqref{1}.

\begin{lemma}\label{lema2}
The real number $\lambda^*$ is an eigenvalue of problem \eqref{1}.
\end{lemma}

\begin{proof}
We first prove that there exists $\in\sob\setminus\{0\}$ such that
$$\lambda^*=\frac{\ee_1(u)}{\ee_2(u)}\,.$$
Let $(u_n)\sob\setminus\{0\}$ be a minimizing sequence of $\ee_1/\ee_2$. By Lemma \ref{lema1}, this sequence is bounded. Next, using the reflexivity of $\sob$, we can assume (up to a subsequence) that
$$u_n\rightharpoonup u\quad\mbox{in}\ \sob.$$
By the lower semicontinuity of $\ee_0$ we obtain
\bb\label{sc1}\ee_0(u)\leq\liminf_{n\to\infty}\ee_0(u_n).\bbb

Next, using hypotheses \eqref{2} and \eqref{3}, we deduce that the function space $\sob$ is continuously embedded into $L^{p_3(x)}(\Omega)$. It follows that we can assume that
$$u_n\ri u\quad\mbox{in}\ L^{p_3(x)}(\Omega).$$
Therefore
\bb\label{sc2}\Theta(u_n)\ri\Theta(u)\quad\mbox{as}\ n\ri\infty.\bbb

Using again compact embeddings of $\sob$ into $L^r(\Omega)$ and $L^s(\Omega)$ we deduce that
\bb\label{sc3} \ee_2(u_n)\ri\ee_2(u)\quad\mbox{as}\ n\ri\infty. \bbb

We now claim that
\bb\label{sc4}u\not=0.  \bbb
Arguing by contradiction, we assume that $u=0$. The previous comments show that
$$\Theta(u_n)\to 0\ \mbox{and}\ \ee_2(u_n)\to 0\quad\mbox{as}\ n\to\infty.$$
Thus, if $\ep>0$ is sufficiently small then for all $n$ large enough we have
$$|\ee_1(u_n)-\lambda^*\ee_2(u_n)|<\ep\ee_2(u_n).$$
This implies $\ee_1(u_n)\to 0$, hence $\ee_0(u_n)\to 0$ as $n\to\infty$. Using (H3) we deduce that
$$u_n\to 0\quad\mbox{in}\ \sob.$$
Thus, by Lemma \ref{lema1}, we deduce that
$$\frac{\ee_1(u_n)}{\ee_2(u_n)}\ri+\infty\quad\mbox{as}\ n\to\infty,$$
hence $\ee_1(u)/\ee_2(u)=+\infty$. This contradiction implies our claim \eqref{sc4}.

Combining relations \eqref{sc1}, \eqref{sc2}, \eqref{sc3}, and \eqref{sc4}, we deduce that
\bb\label{baba}
\frac{\ee_1(u)}{\ee_2(u)}=\inf_{v\in\sob\setminus\{0\}}\frac{\ee_1(v)}{\ee_2(v)}=:\lambda^*.\bbb

Next, we argue that $u$ is an eigenfunction of problem \eqref{1} corresponding to $\lambda=\lambda^*$. For this purpose, we fix $w\in\sob\setminus\{0\}$ and consider the function $h:\RR\to\RR$ defined by
$$h(t)=\frac{\ee_1(u+tw)}{\ee_2(u+tw)}\,.$$
By \eqref{baba} we obtain $h'(0)=0$, hence
\bb\label{baba1}\ee_2(u)\ee_1'(u)(w)-\ee_1(u)\ee_2'(u)(w)=0.\bbb
Recall that $\ee_1(u)=\lambda^*\ee_2(u)$. Thus, relation \eqref{baba1} yields
$$\ee_1'(u)(w)-\lambda^*\ee_2'(u)(w)=0\quad\mbox{for all}\ w\in\sob,$$
hence $u$ is a solution of problem \eqref{1} for $\lambda=\lambda^*$.
\end{proof}

We prove in what follows that every $\lambda>\lambda^*$ is an eigenvalue of problem \eqref{1}, which shows the existence of an unbounded interval of eigenvalues.

\begin{lemma}\label{lema3}
Problem \eqref{1} has a solution for all $\lambda>\lambda^*$.
\end{lemma}

\begin{proof}
Fix $\lambda>\lambda^*$ and consider the map $\nn:\sob\to\RR$ defined by
$$\nn(v)=\ee_1(v)-\lambda\ee_2(v).$$
Then $v$ is a solution of problem \eqref{1} if and only if $v$ is a critical point of $\nn$.

With the same arguments as in the proof of Lemma \ref{lema1} we deduce that $\nn$ is coercive, so it has a global minimum. Since $\nn$ is also lower semicontinuous, we deduce that the global minimum of $\nn$ is attained at some $v\in\sob$.

To conclude the proof it remains to show that $v\not=0$. Using the definition of $\lambda^*$ we deduce that  $\nn(v)<0$, hence $$\inf_{w\in\sob}\nn(w)<0,$$
which implies that $v\not=0$.
\end{proof}

Next, we conclude the proof of Theorem \ref{t1} by proving that the second Rayleigh quotient (which generates the value $\lambda_*$) is not associated with the existence of eigenvalues for problem \eqref{1}.

\begin{lemma}\label{lema4}
Problem \eqref{1} does not have any solution for all $\lambda<\lambda_*$.
\end{lemma}

\begin{proof}
Fix arbitrarily $\lambda<\lambda_*$ and assume by contradiction that $\lambda$ is an eigenvalue for problem \eqref{1}. Thus, there exists $u\in\sob\setminus\{0\}$ such that
$$\begin{array}{ll}
&\di\intom \left((\phi(x,|D  u|)+(\psi(x,|D  u|) \right)DuDvdx+\intom w(x)\theta(x,|  u|)uvdx=\\ &\di
\lambda\intom (|u|^{r-2}+|u|^{s-2})uvdx.\end{array}$$
for all $v\in\sob$.

Taking $v=u$, we obtain that
$$\lambda=\frac{\di \intom (\phi(x,|Du|)+\psi(x,|Du|))|Du|^2dx+\intom w(x)\theta (x,|u|)u^2dx}{\di\intom (|u|^r+|u|^s)dx}\,.$$
This yields a contradiction, since $\lambda<\lambda_*$.
\end{proof}

The proof of Theorem \ref{t1} is now complete.
\qed

\section{Optimization of $\lambda^*$ for bounded closed subsets of $L^\infty(\Omega)$}\label{sect5}
In this section we give the proof of Theorem \ref{t2}. We will denote by $\ee_1^w$ and $\Theta^w$ the energies corresponding to the weight $w$.

Let $\WW$ be a nonempty, bounded, closed subset of $L^\infty(\Omega)$. Let $(w_n)\subset\WW$ be a minimizing sequence of $\inf_{w\in \WW}\lambda^*(w)$, that is,
 \bb\label{start}\inf_{w\in \WW}\lambda^*(w)=\lim_{n\to\infty}\lambda^*(w_n).\bbb
 Since $\WW$ is a bounded set, it follows that the sequence $(w_n)$ is bounded. Thus, by the Banach-Alaoglu-Bourbaki theorem (see Brezis \cite[Theorem 3.16]{hbspringer}) it follows that, up to a subsequence, $w_n\to w_0\in\WW$ in the weak$^*$ topology $\sigma(L^\infty,L^1)$.

 We claim that
\bb\label{wmin} \lambda^*(w_0)=\lim_{n\to\infty}\lambda^*(w_n).\bbb

By Theorem \ref{t1}, there exists $u_n\in\sob\setminus\{0\}$ such that
\bb\label{ab0}\frac{\ee_1^{w_n}(u_n)}{\ee_2(u_n)}=\lambda^*(w_n)\quad\mbox{for all}\ n\geq 1.\bbb
On the other hand, returning to relation \eqref{above}, we have
\bb\label{ab1}\frac{\ee_1^{w_n}(u_n)}{\ee_2(u)}\geq C_0\,\frac{\di \intom (|Du_n|^{p_1(x)}+|Du_n|^{p_2(x)})dx}{\|u_n\|^{r}+\|u_n\|^{s}}-C_4.\bbb
Combining \eqref{ab0} and \eqref{ab1} we deduce that $(u_n)$ is bounded in $\sob$.
Moreover, by Lemma \ref{lema1}, the sequence $(u_n)$ does not contain a subsequence converging to zero. Thus, by reflexivity, there exists $u_0\in\sob\setminus\{0\}$ such that
\bb\label{crai1}u_n\rightharpoonup u_0\quad\mbox{in}\ \sob\bbb
and
\bb\label{crai2}u_n\to u_0\quad\mbox{in}\ L^{p_3(x)}(\Omega).\bbb

By hypothesis (H2), we observe that
$$\begin{array}{ll} |\Theta^{w_0}(u_n)-\Theta^{w_0}(u_0)|&\di =
\left|\intom w_0(x)\int_{|u_0(x)|}^{|u_n(x)|}s\theta(x,s)dsdx\right|\\
&\di\leq \frac{b}{p_3^+} \intom \left(\left|\, |u_n|^{p_3(x)}-|u_0|^{p_3(x)} \right|\right)dx.\end{array}$$
Using now \eqref{crai2}, we obtain
\bb\label{cara1}\Theta^{w_0}(u_n)\to\Theta^{w_0}(u_0)\quad\mbox{as}\ n\to\infty.\bbb

With a similar argument and using the fact that $(w_n)\subset\WW$ and $\WW$ is a bounded set, we deduce that
\bb\label{cara2}\Theta^{w_n}(u_n)\to\Theta^{w_n}(u_0)\quad\mbox{as}\ n\to\infty.\bbb

We now recall that
$w_n\to w_0\in\WW$ in the weak$^*$ topology $\sigma(L^\infty,L^1)$. Since the mapping
$$\Omega\ni x\mapsto \Theta_0^{w_0}(x,|u_0(x)|)$$
is in $L^1(\Omega)$, then Proposition 3.13(iv) in \cite{hbspringer} shows that
\bb\label{cara3}\Theta^{w_n}(u_0)\to\Theta^{w_0}(u_0)\quad\mbox{as}\ n\to\infty.\bbb

Combining relations \eqref{cara1}--\eqref{cara3} we deduce that
\bb\label{cara0}\lim_{n\to\infty}\left[\Theta^{w_n}(u_n)-\Theta^{w_0}(u_n) \right]=0.\bbb

Next, using the definition of $\lambda^*(w_0)$ we have
$$\lambda^*(w_0)\leq\frac{\ee_1^{w_0}(u_n)}{\ee_2(u_n)}\,.$$
Thus, by \eqref{ab0},
\bb\label{spre}\begin{array}{ll}\lambda^*(w_0)&\di\leq\frac{\ee_1^{w_n}(u_n)}{\ee_2(u_n)}+\frac{\ee_1^{w_0}(u_n)
-\ee_1^{w_n}(u_n)}{\ee_2(u_n)}\\
&\di= \lambda^*(w_n)+\frac{\ee_1^{w_0}(u_n)
-\ee_1^{w_n}(u_n)}{\ee_2(u_n)}\\
&\di= \lambda^*(w_n)+\frac{\Theta^{w_0}(u_n)
-\Theta^{w_n}(u_n)}{\ee_2(u_n)}\,.\end{array}\bbb
Relations \eqref{cara0} and \eqref{spre} yield
$$\lambda^*(w_0)\leq\lim_{n\to\infty}\lambda^*(w_n).$$
Combining this relation with \eqref{start} we deduce that
$$\lambda^*(w_0)\leq\inf_{w\in \WW}\lambda^*(w).$$
Since $w_0\in\WW$, this relation shows that, in fact, we have equality, hence our claim \eqref{wmin} is argued. \qed

\subsection*{Concluding remarks and perspectives}
(i) The differential operator that describes problem \eqref{1} falls in the realm of those related to the so-called Musielak-Orlicz spaces (see \cite{19a, orli}), more in general, of the operators having non-standard growth conditions (which are widely considered in the calculus of variations). These function spaces are Orlicz spaces whose defining Young function exhibits an additional dependence on the $x$ variable. Indeed, classical Orlicz spaces $L^\Phi$ are defined requiring that a member function $f$ satisfies
$$\int_\Omega \Phi(|f|)dx<\infty,$$
where $\Phi(t)$ is a Young function (convex, non-decreasing, $\Phi(0)=0$). In the new case of Musielak-Orlicz spaces, the above condition becomes
$$\int_\Omega \Phi(x,|f|)dx<\infty.$$

In a particular case, the problems considered in this paper are driven by the function
$$\Phi(x,|\xi|):=\left\{
\begin{array}{lll}
& |\xi|^{p_1(x)}&\quad\mbox{if}\ |\xi|\leq 1\\
& |\xi|^{p_2(x)}&\quad\mbox{if}\ |\xi|\geq 1.
\end{array}\right.
$$

(ii) According to the recent papers by Baroni, Colombo and Mingione \cite{mingi1,mingi2,mingi3} dealing with double phase energy functional associated to $(p,q)$-operators,
we consider that an interesting field of research concerns nonhomogeneous problems of the type
\bb\label{iint1}\intom \left[\Phi_0(x,|Du(x)|)+a(x)\Psi_0(x,|Du(x)|)\right]dx\bbb
or
\bb\label{iint2}\intom \left[\Phi_0(x,|Du(x)|)+a(x)\Psi_0(x,|Du(x)|)\log(e+|x|)\right]dx,\bbb
where $a(x)$ is a nonnegative potential and the assumption \eqref{2} holds.

In such a case, the variable
potential controls the geometry of a composite of two materials described by $\Phi_0$ and $\Psi_0$, hence by $\phi$ and $\psi$. In the region $[x:\ a(x)>0]$ then the material described by $\psi$ is present, otherwise the material described by $\phi$ is the only one that creates the composite.

We also point out
that  since  the  integral  energy  functionals  defined  in \eqref{iint1} and \eqref{iint2} have  a  degenerate  behavior  on  the
zero  set  of  the  gradient,  it  is  natural  to  study  what  happens  if  the  integrand  is  modified  in such a way that, also if $|Du|$ is small, there exists an imbalance between the two terms of every
integrand.

In a related framework, we refer to our recent paper Bahrouni, R\u{a}dulescu, and Repov\v{s} \cite{nonln}, where it is studied a class of integral functions with variable exponent and vanishing weight and it is established a Caffarelli-Kohn-Nirenberg inequality in this {\it degenerate} setting.

(iii)
The problem studied in the present paper corresponds to a {\it subcritical} setting, which is described in hypothesis \eqref{3}. This assumption has been crucial in our arguments in order to deal with strong convergence in suitable Lebesgue spaces with variable exponent. Due to the particular setting existing in this paper, we suggest the study of  ``almost critical" abstract framework, which assumes to replace hypothesis \eqref{3} with
$$\min\{ p_1^*(x)-p_2(x);\ x\in\overline\Omega\}=0.$$
In particular, this assumption holds
if there exists $x_0\in\Omega$ such that $$p_2(x_0)=p_1^*(x_0)\ \mbox{and}\ p_2(x)< p_1^*(x)\ \mbox{for all
$x\in\overline\Omega\setminus\{x_0\}$}.$$ We do not have any response to this new setting, even for particular cases (for instance, if $\Omega$ is a ball).

\medskip
{\bf Acknowledgements.} The authors thank an anonymous referee for the careful analysis of this paper and for suggesting several improvements with respect to the initial version of this work.
This research was supported by the Slovenian Research Agency grants
P1-0292, J1-8131, J1-7025, N1-0064, and N1-0083. V.D.~R\u adulescu acknowledges the support through a grant of the Romanian Ministry of Research and Innovation, CNCS--UEFISCDI, project number PN-III-P4-ID-PCE-2016-0130,
within PNCDI III.


\begin{thebibliography}{99}
\bibitem{hal} T.C. Halsey, Electrorheological fluids, {\it Science}
{\bf 258} (1992), 761-766.

\bibitem{R} M. R\r{u}\v{z}i\v{c}ka, {\it Electrorheological Fluids: Modeling
and Mathematical Theory}, Springer--Verlag, Berlin, 2002.

\bibitem{CLR} Y. Chen, S. Levine, and M. Rao, Variable exponent, linear growth functionals in image
processing, {\it SIAM J.~Appl. Math.} {\bf 66} (2006), 1383-1406.

\bibitem{marce1} P. Marcellini, On the definition and the lower semicontinuity of certain quasiconvex integrals, {\it Ann. Inst. H. Poincar\'e, Anal. Non Lin\'eaire} {\bf 3} (1986), 391-409.

\bibitem{marce2} P. Marcellini, Regularity and existence of solutions of elliptic equations with $p, q$--growth conditions, {\it J.
Differential Equations} {\bf 90} (1991), 1-30.

\bibitem{marce3} P. Marcellini, Everywhere regularity for a class of elliptic systems without growth conditions, {\it Ann. Scuola Norm. Sup. Pisa Cl. Sci.} {\bf (4) 23} (1996), no. 1, 1-25.

\bibitem{ball1} J.M. Ball, Convexity conditions and existence theorems in nonlinear elasticity, {\it
Arch. Rational Mech. Anal.} {\bf 63} (1976/77), no. 4, 337-403.

\bibitem{ball2} J.M. Ball, Discontinuous equilibrium solutions and cavitation in nonlinear elasticity, {\it
Philos. Trans. Roy. Soc. London Ser. A} {\bf 306} (1982), no. 1496, 557-611.

\bibitem{fusco} N. Fusco and C. Sbordone, Some remarks on the regularity of minima of anisotropic integrals, {\it Comm. Partial Differential Equations} {\bf 18} (1993), no. 1-2, 153-167.

\bibitem{morrey} Ch.B. Morrey, Jr., {\it Multiple Integrals in the Calculus of Variations}, Reprint of the 1966 edition. Classics in Mathematics, Springer-Verlag, Berlin, 2008.

    \bibitem{mingi1} P. Baroni, M. Colombo, and G. Mingione, Harnack inequalities for double phase functionals, {\it Nonlinear Anal.} {\bf 121} (2015), 206-222.

\bibitem{mingi2} P. Baroni, M. Colombo, and G. Mingione, Nonautonomous functionals, borderline cases and related
function classes, {\it St. Petersburg Math. J.} {\bf 27} (2016), no. 3, 347-379.

\bibitem{mingi3} P. Baroni, M. Colombo, and G. Mingione, Regularity for general functionals with double phase, {\it Calc. Var.}  (2018), 57:62 {\tt https://doi.org/10.1007/s00526-018-1332-z}.

\bibitem{mingi4} M. Colombo and G. Mingione, Regularity for double phase variational problems, {\it Arch. Ration. Mech. Anal.} {\bf 215} (2015), no. 2, 443-496.

\bibitem{mingi5} M. Colombo and G. Mingione, Bounded minimisers of double phase variational integrals, {\it  Arch. Ration. Mech. Anal.} {\bf 218} (2015), no. 1, 219-273.

\bibitem{zhikov1} V.V. Zhikov, Averaging of functionals of the calculus of variations and elasticity theory, {\it Izv. Akad. Nauk SSSR Ser. Mat.} {\bf 50} (1986), no. 4, 675-710; English translation in {\it Math. USSR-Izv.} {\bf 29} (1987), no. 1, 33-66.

\bibitem{zhikov2} V.V. Zhikov, On Lavrentiev's phenomenon, {\it Russian J. Math. Phys.} {\bf 3} (1995), no. 2, 249-269.

\bibitem{zhikov3} V.V. Zhikov, On the density of smooth functions in a weighted Sobolev space, {\it Dokl. Akad. Nauk} {\bf 453} (2013), no. 3, 247-251; English translation in {\it Dokl. Math.} {\bf 88} (2013), no. 3, 669-673.

\bibitem{acerbi2} E. Acerbi and G. Mingione, Gradient estimates for the $p(x)$-Laplacean system,
{\it J.~Reine Angew. Math.} {\bf 584} (2005), 117-148.

\bibitem{colombo} M. Colombo and G. Mingione, Calder\'on-Zygmund estimates and non-uniformly elliptic operators, {\it J. Funct. Anal.} {\bf 270} (2016), no. 4, 1416-1478.

\bibitem{diening} L. Diening, P. Harjulehto, P. H\"{a}st\"{o}, and M. R\r{u}\v{z}i\v{c}ka, \textit{Lebesgue and Sobolev Spaces with Variable Exponents},
Lecture Notes in Mathematics, vol. 2017, Springer-Verlag, Berlin, 2011.

\bibitem{radrep} V. R\u adulescu and D. Repov\v{s}, {\it  Partial Differential Equations with Variable Exponents: Variational Methods and Qualitative Analysis}, CRC Press, Taylor \& Francis Group, Boca Raton FL, 2015.   
   
\bibitem{cristina}  C. De Filippis, Higher integrability for constrained minimizers of integral functionals with $(p,q)$-growth in low dimension, {\it Nonlinear Anal.} {\bf 170} (2018), 1-20.   
  
\bibitem{esposito}  L. Esposito, F. Leonetti, and G. Mingione, Sharp regularity for functionals with $(p,q)$-growth, {\it J. Differential Equations} {\bf 204} (2004), 5-55.
    

\bibitem{dou1} V. R\u adulescu and Q. Zhang, Double phase anisotropic variational problems and combined effects of reaction and absorption terms, {\it J. Math. Pures Appl.}, to appear.

\bibitem{dou2} X. Shi, V. R\u adulescu, D. Repov\v{s},  and Q. Zhang, Multiple solutions of double phase variational problems with variable exponent, {\it Advances in the Calculus of Variations}, to appear.

\bibitem{rjam} M. Mih\u{a}ilescu and V. R\u{a}dulescu,  Concentration phenomena in nonlinear eigenvalue problems with variable exponents and sign-changing potential, {\it J. Anal. Math.} {\bf 111} (2010), 267-287.

\bibitem{chorfi} S. Baraket, S. Chebbi, N. Chorfi, and V. R\u adulescu,  Non-autonomous eigenvalue problems with variable $(p_1,p_2)$-growth, {\it Adv. Nonlinear Stud.} {\bf 17} (2017), no. 4, 781-792.

\bibitem{cite7} X.L. Fan, D. Zhao, On the spaces $L^{p(x)}\left( \Omega \right) $
and $W^{m,p(x)}\left( \Omega \right) $, \textit{J. Math. Anal. Appl.}
\textbf{263} (2001), 424-446.

\bibitem{radnla} V. R\u adulescu, Nonlinear elliptic equations with variable exponent: old and new, {\it Nonlinear Analysis: Theory, Methods and Applications} {\bf 121} (2015), 336-369.

\bibitem{edm2} D.E. Edmunds and J. R\'akosn\'{\i}k, Density of smooth
functions in $W^{k,p(x)}(\Omega)$, {\it Proc. Roy. Soc. London Ser.~A}
{\bf 437} (1992), 229-236.

\bibitem{edm} D.E. Edmunds, J. Lang, and A. Nekvinda, On $L^{p(x)}$
norms, {\it Proc. Roy. Soc. London Ser.~A} {\bf 455} (1999), 219-225.

\bibitem{edm3} D.E. Edmunds and J. R\'akosn\'{\i}k, Sobolev embedding
with variable exponent, {\it Studia Math.} {\bf 143} (2000),
267-293.

\bibitem{kim} I.H. Kim and Y.H. Kim, Mountain pass type solutions and positivity of the infimum eigenvalue for quasilinear elliptic equations with variable exponents, {\it Manuscripta Math.} {\bf 147} (2015), 169-191.

\bibitem{simon} J. Simon,  R\'egularit\'e de la solution d'une \'equation non lin\'eaire dans $\RR^N$,  {\it Journ\'ees d'Analyse Non Lin\'eaire (Proc. Conf., Besan\c{c}on, 1977)}, pp. 205-227, Lecture Notes in Math., 665, Springer, Berlin, 1978.

\bibitem{fpr2008} R. Filippucci, P. Pucci, and V. R\u adulescu, Existence and non-existence results for quasilinear elliptic exterior problems with nonlinear boundary conditions, {\it Comm. Partial Differential Equations} {\bf 33} (2008), no. 4-6, 706-717.

\bibitem{fucik} S. Fu\v{c}ik, J. Ne\v{c}as, J. Sou\v{c}ek, and V. Sou\v{c}ek, {\it Spectral Analysis of Nonlinear Operators}, Lecture Notes in Mathematics, Vol. 346, Springer-Verlag, Berlin-New York, 1973.    

\bibitem{cola} F. Colasuonno and M. Squassina,  Eigenvalues for double phase variational integrals, {\it Ann. Mat. Pura Appl.} {\bf (4) 195} (2016), no. 6, 1917-1959.

\bibitem{hbspringer} H. Brezis, {\it Functional Analysis, Sobolev Spaces and Partial Differential Equations}, Universitext, Springer, New York, 2011.

\bibitem{19a} J. Musielak, {\it Orlicz Spaces and Modular Spaces}, Lecture
Notes in Math. 1034, Springer-Verlag, Berlin, 1983.

\bibitem{orli} W. Orlicz, \"Uber konjugierte Exponentenfolgen, {\it Studia Math.} {\bf
3} (1931),  200-211.

\bibitem{nonln} A. Bahrouni, V.D. R\u adulescu, and D.D. Repov\v{s}, A weighted anisotropic variant of the Caffarelli-Kohn-Nirenberg inequality and applications, {\it Nonlinearity} {\bf 31} (2018), 1518-1534.
\end{thebibliography}
\end{document}